\documentclass{cimart}

\usepackage{dsfont}

\title{
    Quantized collision invariants on the sphere
    }

\author{
    Benjamin Anwasia and Diogo Ars\'enio
    }

\authorinfo[
    B. Anwasia]{
    New York University Abu Dhabi, Abu Dhabi, United Arab Emirates}{%
    benjamin.anwasia@nyu.edu
    }

\authorinfo[
    D. Ars\'enio]{
    New York University Abu Dhabi, Abu Dhabi, United Arab Emirates}{%
    diogo.arsenio@nyu.edu
    }

\abstract{%
	We show that a measurable function $g:\mathbb{S}^{d-1}\to\mathbb{R}$, with $d\geq 3$, satisfies the functional relation
	\begin{equation*}
		g(\omega)+g(\omega_*)=g(\omega')+g(\omega_*'),
	\end{equation*}
	for all admissible $\omega,\omega_*,\omega',\omega_*'\in\mathbb{S}^{d-1}$ in the sense that
	\begin{equation*}
		\omega+\omega_*=\omega'+\omega_*',
	\end{equation*}
	if and only if it can be written as
	\begin{equation*}
		g(\omega)=A+B\cdot\omega,
	\end{equation*}
	for some constants $A\in \mathbb{R}$ and $B\in\mathbb{R}^d$.
	
	Such functions form a family of quantized collision invariants which play a fundamental role in the study of hydrodynamic regimes of the Boltzmann--Fermi--Dirac equation near Fermionic condensates, i.e., at low temperatures. In particular, they characterize the elastic collisional dynamics of Fermions near a statistical equilibrium where quantum effects are predominant.
    }



\VOLUME{32}
\YEAR{2024}
\ISSUE{3}
\firstpage{229}
\DOI{https://doi.org/10.46298/cm.12766}

\begin{document}

\section{Introduction}

Our goal is to fully determine the set of all Borel measurable functions
\begin{equation*}
	g:\partial B(0,R)\to\mathbb{R},
\end{equation*}
where $\partial B(0,R)\subset\mathbb{R}^d$, with $d\geq 2$, for some fixed value $R>0$, which satisfy the functional equation
\begin{equation}\label{quantum:collision:invariant}
	g(\omega)+g(\omega_*)=g(\omega')+g(\omega_*'),
\end{equation}
for all $\omega,\omega_*,\omega',\omega_*'\in \partial B(0,R)$ such that
\begin{equation}\label{conservation:momentum}
	\omega+\omega_*=\omega'+\omega_*'.
\end{equation}
We assume here that the sphere $\partial B(0,R)$ is equipped with its standard surface measure on Borel sets.

There are several possible ways of writing the variables $\omega'$ and $\omega_*'$ in terms of  $\omega$ and $\omega_*$; one of which is given by the formulas
\begin{equation*}
	\omega'=\frac{\omega+\omega_*}2+\frac{|\omega-\omega_*|}2{\bf n},
	\qquad
	\omega_*'=\frac{\omega+\omega_*}2-\frac{|\omega-\omega_*|}2{\bf n},
\end{equation*}
where ${\bf{n}}$ is a unit vector which is orthogonal to $\frac{\omega+\omega_*}2$, if $\omega+\omega_*\neq 0$.

This problem comes up naturally in the study of hydrodynamic limits of the Boltzmann-Fermi-Dirac equation at low temperatures, i.e., near a Fermionic condensate, which we briefly present below. In this context, the function $g(\omega)$ is a fluctuation of the distribution of particles with velocity $\omega$ and the functional equation \eqref{quantum:collision:invariant} is an expression of the property that fluctuations have reached a quantized thermodynamic equilibrium.

Specifically, in such a hydrodynamic regime, the particles which are responsible for the dynamics of the condensate are all at Fermi energy $|\omega|=R$ (defined as the difference between the highest and lowest admissible energy levels). In a collision between two such particles with ingoing velocities $\omega,\omega_*\in\partial B(0,R)$ and outgoing velocities $\omega',\omega_*'\in\partial B(0,R)$, the momentum is conserved, as expressed by \eqref{conservation:momentum}. Collisions are therefore elastic. However, they are quantized in the sense that particles remain at the Fermi energy level.

The solution to this problem, which is given in Proposition \ref{prop:quantum:2d} and Theorem \ref{thm:quantum}, below, is therefore of capital importance in the understanding of the macroscopic behavior of Fermi gases near their absolute zero temperature.

\section[Hydrodynamic regimes of the Boltzmann-Fermi-Dirac equation]{Hydrodynamic regimes of the Boltzmann--Fermi--Dirac equation at low temperatures}

In \cite{aa24bis}, we initiate the analysis of hydrodynamic limits of the Boltzmann--Fermi--Dirac equation, in a low-temperature regime, by studying its acoustic limit. More precisely, we consider solutions $f_\varepsilon(t,x,v)\geq 0$, with $(t,x,v)\in\mathbb{R}^+\times\mathbb{R}^d\times\mathbb{R}^d$, of the Boltzmann--Fermi--Dirac equation
\begin{equation*}
	(\partial_t+v\cdot\nabla_x)f_\varepsilon
	=\frac 1{\varepsilon^\kappa}Q_{BFD}(f_\varepsilon),
\end{equation*}
where $\varepsilon>0$ and $\kappa>0$ determine the rate of convergence of the Knudsen number $\varepsilon^\kappa$.

For simplicity, we omit discussing the explicit format of the Boltzmann--Fermi--Dirac operator $Q_{BFD}$. We will only mention here that it accounts for the elastic collisions between Fermions in the gas, in consistency with the fact that such particles satisfy Pauli's exclusion principle. Furthermore, its structure provides conservation properties of mass, momentum, energy and entropy (i.e., the solutions satisfy an $H$-theorem).

The existence and uniqueness of weak solutions to this equation was established by {Dolbeault} in \cite{d94} under simple and natural assumptions. The construction of solutions in a more general setting was then addressed by {Lions} in \cite{l94}. However, one should note that the solutions constructed in \cite{d94} satisfy more conservation properties than the ones from \cite{l94}.

For convenience, the solutions considered in \cite{aa24bis} are set in the framework provided by the results from \cite{d94}. Their initial data $f_\varepsilon^\mathrm{in}$ are assumed to satisfy a natural uniform bound on their relative entropy
\begin{equation*}
	\frac 1{\varepsilon^{2-\tau}}H(f_\varepsilon^\mathrm{in}|M_\varepsilon)=
	\frac 1{\varepsilon^{2-\tau}}\int_{\mathbb{R}^d\times\mathbb{R}^d}
	\bigg(\underbrace{f_\varepsilon^\mathrm{in}\log\frac{f_\varepsilon^\mathrm{in}}{M_\varepsilon}
	+(1-f_\varepsilon^\mathrm{in})\log\frac{1-f_\varepsilon^\mathrm{in}}{1-M_\varepsilon}}_{\geq 0}\bigg)
	dxdv\leq C^\mathrm{in}<\infty,
\end{equation*}
where $\tau>0$ and $M_\varepsilon$ is the Fermi--Dirac distribution given by
\begin{equation*}
	M_\varepsilon(v)=\frac 1{1+\exp\left(\frac{|v|^2-R^2}{\varepsilon^\tau}\right)}.
\end{equation*}
Note that, in the low temperature regime $\varepsilon\to 0$, the density distribution $M_\varepsilon$ converges pointwise to the Fermionic condensate
\begin{equation*}
	F(v)=\left\{
	\begin{aligned}
		&0 &&\text{if }|v|>R,
		\\
		&\frac 12 &&\text{if }|v|=R,
		\\
		&1 &&\text{if }|v|<R.
	\end{aligned}
	\right.
\end{equation*}
Thus, in order to study hydrodynamic regimes near Fermionic condensates, we consider the density fluctuations $g_\varepsilon(t,x,v)$ given by
\begin{equation*}
	f_\varepsilon=M_\varepsilon+\varepsilon g_\varepsilon.
\end{equation*}

Now, supposing that $0<\tau<1$ and $\kappa>2\tau$, and imposing only natural assumptions on the Boltzmann--Fermi--Dirac operator $Q_{BFD}$, a simplified version of the main results from \cite{aa24bis} contains the statement,
under the above uniform bound on the initial relative entropy, that the fluctuations $g_\varepsilon$ are uniformly bounded in $L^\infty(dt;L^1_\mathrm{loc}(dx;L^1((1+|v|^2)dv)))$ and, up to extraction of a subsequence, converge toward a limit point
\begin{equation*}
	\mu(t,x,v)=g\left(t,x,R\frac{v}{|v|}\right)
	dt\otimes dx\otimes \delta_{\partial B(0,R)}(v),
\end{equation*}
in the weak* topology of locally finite Radon measures. In particular, observe that the fluctuations $g_\varepsilon$ undergo a velocity concentration phenomenon on the sphere $\partial B(0,R)$. Furthermore, the limiting density $g(t,x,\omega)$ belongs to $L^\infty(dt;L^2(\mathbb{R}^d\times\partial B(0,R)))$ and is a quantum collision invariant, in the sense that it satisfies the functional equation \eqref{quantum:collision:invariant}.

This convergence result plays a fundamental role in the study of hydrodynamic regimes of Fermi gases near absolute zero and it provides the context in which the functional equation \eqref{quantum:collision:invariant} arises. The characterization of all quantum collision invariants is thus a crucial step toward a complete understanding of the macroscopic behavior of Fermionic condensates.

\section{Classical collision invariants}

The classical Boltzmann equation and its hydrodynamic regimes feature an analogous problem. Thus, we now give a brief overview of the classical version of collision invariants.

In the classical setting, collision invariants are made up of all Borel measurable functions $g:\mathbb{R}^d\to\mathbb{R}$, with $d\geq 2$, which satisfy the functional equation
\begin{equation*}
	g(v)+g(v_*)=g(v')+g(v_*'),
\end{equation*}
for all $v,v_*,v',v_*'\in \mathbb{R}^d$, such that
\begin{equation*}
	\begin{aligned}
		v+v_*&=v'+v_*',
		\\
		|v|^2+|v_*|^2&=|v'|^2+|v_*'|^2.
	\end{aligned}
\end{equation*}
Here, one of the several possible ways of representing $v'$ and $v'_*$ in terms of $v$ and $v_*$ is given by the following formulas
\begin{equation*}
	v'=\frac{v+v_*}2+\frac{|v-v_*|}2{\bf n},
	\qquad
	v_*'=\frac{v+v_*}2-\frac{|v-v_*|}2{\bf n},
\end{equation*}
where ${\bf{n}}$ is a unit vector.

It turns out that classical collision invariants comprise the kernel of the linearized Boltzmann collision operator. Furthermore, they appear naturally in the derivation of hydrodynamic limits of the Boltzmann equation where they characterize the thermodynamic equilibria of limiting density fluctuations. We refer to \cite{sr09} for a comprehensive survey of hydrodynamic limits of the Boltzmann equation.

In that classical case, it is possible to show that $g(v)$ is a collision invariant if and only if there are $A,C\in\mathbb{R}$ and $B\in\mathbb{R}^d$ such that
\begin{equation*}
	g(v)=A+B\cdot v+C|v|^2,
\end{equation*}
for all $v\in\mathbb{R}^d$.

This complete characterization of classical collision invariants has been known for a long time, at least in regular settings. Indeed, the first analysis of collision invariants was given by {Boltzmann} himself under the assumption that $g(v)$ is twice differentiable. Several other proofs have subsequently been found. Insightful historical accounts of the mathematical developments which led to the above result can be found in \cite{ac90} and \cite[Section 3.1]{cip94}.

It is to be emphasized that all existing proofs in the classical setting rely on the vector structure of the Euclidean space $\mathbb{R}^d$ and fail to carry over to the quantum setting. Nevertheless, some classical proofs are based on a reduction of the problem at hand to the study of Cauchy's functional equation, which will also be useful in the quantum case. Therefore, in the next section, we recall and discuss the main results concerning the resolution of Cauchy's functional equation.

\section{Cauchy's functional equation}

The problem of finding all continuous additive functions on the real line was first considered by Cauchy. A useful version of Cauchy's result is phrased in the following theorem.

\begin{theorem}[Cauchy's functional equation]
	Let $h:[-a,a]\to \mathbb{R}$, for some $a>0$, be such that
	\begin{equation*}
		h(x+y)=h(x)+h(y),
	\end{equation*}
	for all $x,y\in[-a,a]$, with $x+y\in [-a,a]$. If $h$ is continuous at one point $x_0\in[-a,a]$, then there exists a constant $c\in\mathbb{R}$ such that $h(x)=cx$, for all $x\in[-a,a]$.
\end{theorem}

The proof of the above result is by now a standard exercise often featured in modern courses on mathematical analysis. It consists in establishing first that $h$ is $\mathbb{Q}$-linear, in the sense that $h(\alpha x)=\alpha h(x)$, for all $\alpha\in\mathbb{Q}$ and $x\in[-a,a]$ such that $\alpha x\in [-a,a]$, and then extending this property to all $\alpha\in\mathbb{R}$ by exploiting the continuity of $h$ at one point.

This result is essential in the characterization of collision invariants. However, in order to reach results which are more broadly applicable, it is now important to extend it to all measurable functions. Historically, such extensions have attracted considerable interest, as illustrated by the articles \cite{banach20}, \cite{sierpinski20} and \cite{steinhaus20}, and are a consequence of the automatic continuity of homomorphisms. We refer to \cite{r09} for a modern survey of results on automatic continuity.

For the sake of completeness, we give now a brief presentation of Steinhaus's approach to automatic continuity. We begin by recalling an important lemma on the topology of measurable sets known as Steinhaus's Theorem.

\begin{lemma}[Steinhaus's Theorem, \cite{steinhaus20}]
	Let $A\subset\mathbb{R}^d$, with $d\geq 1$, be a Borel measurable set of positive measure. Then, the set
	\begin{equation*}
		A-A=\{x-y\in\mathbb{R}^d : x,y\in A\}
	\end{equation*}
	is a neighborhood of the origin.
\end{lemma}

We give here a simple justification of this result. This modern elementary proof is standard and is part of mathematical folklore. Unfortunately, we do not know who should be credited for its original idea. Another short and insightful justification of Steinhaus's Theorem can be found in \cite{kestelman47} and in the proof of Theorem 2.3 in \cite{r09}.

\begin{proof}
	We consider the convolution
	\begin{equation*}
		f(x)=\mathds{1}_A * \mathds{1}_{-A}(x)
		=\int_{\mathbb{R}^d}\mathds{1}_A(y)\mathds{1}_A(x+y)dy,
	\end{equation*}
	which is a continuous function. It is readily seen that, if $f(x)>0$, there must exist $y\in A$ such that $x+y\in A$, thereby implying that $x=(x+y)-y\in A-A$. In other words, if $x\notin A-A$, then $f(x)=0$.
	
	In particular, noticing that $f(0)=|A|>0$, we deduce that $\{f(x)>\frac{|A|}2\}$ is an open neighborhood of the origin which is contained in $A-A$. This completes the proof of the lemma.
\end{proof}

An application of Steinhaus's Theorem leads to the automatic continuity of homomorphisms. This is the content of the next lemma.

\begin{lemma}[Automatic continuity of homomorphisms]
	Let $h:[-a,a]\to \mathbb{R}$, for some $a>0$, be Borel measurable and such that
	\begin{equation*}
		h(x+y)=h(x)+h(y),
	\end{equation*}
	for all $x,y\in[-a,a]$, with $x+y\in [-a,a]$. Then, it is continuous at the origin.
\end{lemma}

We give a brief justification of this lemma based on an adaptation of the proof of Theorem 1 from \cite{kestelman47} (see also the proof of Theorem 2.2 in \cite{r09}). As pointed out in \cite{kestelman47}, this lemma remains valid if the measurability assumption on $h$ is weakened by only requiring that $h$ be bounded on a set of positive measure.

\begin{proof}
	Let $O\subset\mathbb{R}$ be an open set containing the origin and take another open set $U\subset O$ containing the origin such that $U-U\subset O$. Since $\cup_{q\in\mathbb{Q}}h^{-1}(q+U)=[-a,a]$, there must exist $q_0\in\mathbb{Q}$ such that $V=h^{-1}(q_0+U)$ is of positive measure. Therefore, by Steinhaus's Theorem, the measurable set $W=(V-V)\cap [-a,a]$ is a neighborhood of the origin. Moreover, since $h$ is additive, we notice that $h(W)\subset h(V)-h(V)\subset U-U\subset O$. It therefore follows that $h$ is continuous at the origin.
\end{proof}

It is readily seen that the combination of the preceding results yields Cauchy's result for measurable functions, as stated in the next corollary. We will make use of this fundamental result in the proof of Theorem \ref{thm:quantum}, below.

\begin{corollary}[Cauchy's functional equation for measurable functions]\label{cauchy:measurable}
	Let $h:[-a,a]\to \mathbb{R}$, for some $a>0$, be Borel measurable and such that
	\begin{equation*}
		h(x+y)=h(x)+h(y),
	\end{equation*}
	for all $x,y\in[-a,a]$, with $x+y\in [-a,a]$. Then, there exists a constant $c\in\mathbb{R}$ such that $h(x)=cx$, for all $x\in[-a,a]$.
\end{corollary}

\section{Quantum collision invariants}\label{section:quantum}

As previously mentioned, there are several available proofs of the characterization of collision invariants in the classical Euclidean setting $\mathbb{R}^d$. However, these methods depend crucially on the vector structure of the Euclidean space $\mathbb{R}^d$ and, therefore, cannot be directly adapted to the spherical setting of quantum collision invariants. We present now a new approach which allows us to take into account the spherical geometry of quantized collisions and give a complete description of the associated collision invariants.

\subsection{In two dimensions}

It turns out that the resolution of the functional equation \eqref{quantum:collision:invariant} in the two-dimensional setting differs from the higher-dimensional case. Therefore, we treat these settings separately and start by providing a complete description of the case $d=2$.

\begin{proposition}\label{prop:quantum:2d}
	Let $d=2$ and consider a Borel measurable function $g:\partial B(0,R)\to\mathbb{R}$, for some fixed $R>0$. Suppose that $g$ is a collision invariant in the sense that it solves \eqref{quantum:collision:invariant}, for all admissible $\omega,\omega_*,\omega',\omega_*'\in \partial B(0,R)$ such that \eqref{conservation:momentum} holds true.
	
	Then, there is $A\in\mathbb{R}$ such that
	\begin{equation*}
		g(\omega)+g(-\omega)=A,
	\end{equation*}
	for all $\omega\in \partial B(0,R)$.
\end{proposition}

\begin{proof}
	We begin by determining which velocities are admissible. To that end, suppose first that $\omega_*\neq -\omega$. In that case, it is readily seen that $(\omega',\omega_*')=(\omega,\omega_*)$ or $(\omega',\omega_*')=(\omega_*,\omega)$, whereby \eqref{quantum:collision:invariant} is automatically satisfied for any function $g$. There is therefore no helpful information to extract from that case.
	
	The case $\omega_*=-\omega$ is more interesting. It yields that
	\begin{equation*}
		g(\omega)+g(-\omega)=g(\omega')+g(-\omega'),
	\end{equation*}
	for all velocities $\omega$ and $\omega'$ on the circle $\partial B(0,R)$. This establishes that the even part of $g$ is constant, which concludes the proof.
\end{proof}

\begin{remark}
	Notice that, in the two-dimensional case, any function with a constant even part is a quantum collision invariant. In particular, two-dimensional collision invariants have no constraint on their odd part.
\end{remark}

\subsection{In three and higher dimensions}

We give now a complete characterization of solutions to the functional equation \eqref{quantum:collision:invariant} in dimensions $d\geq 3$.

\begin{theorem}\label{thm:quantum}
	Let $d\geq 3$ and consider a Borel measurable function $g:\partial B(0,R)\to\mathbb{R}$, for some fixed $R>0$. Suppose that $g$ is a collision invariant in the sense that it solves \eqref{quantum:collision:invariant}, for all admissible $\omega,\omega_*,\omega',\omega_*'\in \partial B(0,R)$ such that \eqref{conservation:momentum} holds true.
	
	Then, there are $A\in\mathbb{R}$ and $B\in\mathbb{R}^d$ such that
	\begin{equation*}
		g(\omega)=A+B\cdot\omega,
	\end{equation*}
	for all $\omega\in \partial B(0,R)$.
\end{theorem}

\begin{remark}
	Note that any function on the sphere $\partial B(0,R)$ of the form $g(\omega)=A+B\cdot\omega$ solves the functional equation \eqref{quantum:collision:invariant}. Therefore, the preceding theorem gives a full description of quantum collision invariants.
\end{remark}

\begin{proof}
	By considering $g(R\omega)$, with $\omega\in\mathbb{S}^{d-1}$, instead of $g(\omega)$, with $\omega\in\partial B(0,R)$, we first reduce the problem to the case $R=1$, which will allow for a more convenient use of notation.
	
	Now, for any set of quantized velocities $\omega,\omega_*,\omega',\omega_*'\in\mathbb{S}^{d-1}$, which are admissible in the sense that they satisfy \eqref{conservation:momentum}, we observe that fixing $i\in\{1,\ldots,d\}$ and changing the sign of the $i$th coordinate of each vector produces a new set of admissible quantized vectors. Indeed, this follows from the elementary facts that \eqref{conservation:momentum} is a coordinatewise relation and that the norm of a vector does not detect a change of sign in any coordinate.
	
	This simple observation allows us to uniquely decompose $g$ into a sum of collision invariants
	\begin{equation}\label{decomposition:1}
		g=\sum_{n=0}^d
		\sum_{\substack{I=\{i_1,i_2,\ldots,i_n\}
		\\
		1\leq i_1<i_2<\ldots<i_n\leq d}}g_I,
	\end{equation}
	where $g_I(\omega)$ is odd in each coordinate $\omega_i$, if $i\in I$, and even in $\omega_i$, if $i\in \{1,\ldots, d\}\setminus I$. (Note that it is implied that $I=\emptyset$ when $n=0$.) This is achieved by applying successive decompositions of $g$ into its odd and even components, in each coordinate.
	
	To be more precise, one can also give the explicit formula
	\begin{equation*}
		g_I(\omega)=2^{-d}\sum_{\sigma\in\{1,-1\}^d}
		\sigma_{i_1}\sigma_{i_2}\cdots\sigma_{i_n}g(\sigma_1\omega_1,\sigma_2\omega_2,\ldots,\sigma_d\omega_d),
	\end{equation*}
	which clearly expresses $g_I$ as a linear combination of collision invariants, thereby showing that $g_I$ is itself a collision invariant. In particular, it is then readily seen that $g_I$ has the required symmetry properties in each coordinate.
	
	We are now going to characterize each $g_I$ for different values of $n\in\{0,1,\ldots,d\}$. First of all, if $n\geq 2$, we claim that $g_I=0$. To see this, suppose, without loss of generality and for mere convenience of notation, that $\{1,2\}\subset I$. Then, since $g_I$ is odd in each coordinate $\omega_1$ and $\omega_2$, we deduce that it is even in $(\omega_1,\omega_2)$, i.e.,
	\begin{equation*}
		g_I(-\omega_1,-\omega_2,\widetilde \omega)=g_I(\omega_1,\omega_2,\widetilde \omega),
	\end{equation*}
	where $\widetilde\omega=(\omega_3,\ldots,\omega_d)$ denotes the remaining coordinates. Therefore, since $g_I$ is a collision invariant, we find that
	\begin{equation*}
		\begin{aligned}
			2g_I(\omega_1,\omega_2,\widetilde \omega)
			&=g_I(\omega_1,\omega_2,\widetilde \omega)+g_I(-\omega_1,-\omega_2,\widetilde \omega)
			\\
			&=g_I(-\omega_1,\omega_2,\widetilde \omega)+g_I(\omega_1,-\omega_2,\widetilde \omega)
			=-2g_I(\omega_1,\omega_2,\widetilde \omega),
		\end{aligned}
	\end{equation*}
	which necessarily implies that $g_I(\omega)=0$, for all $\omega\in\mathbb{S}^{d-1}$.
	
	This allows us to reduce \eqref{decomposition:1} to the simpler decomposition
	\begin{equation}\label{decomposition:2}
		g=g_0+\sum_{i=1}^dg_i,
	\end{equation}
	where we denote $g_0=g_{\emptyset}$, which corresponds to the case $n=0$, and $g_i=g_{\{i\}}$, which corresponds to the case $n=1$.
	
	Next, since $g_0$ is a collision invariant which is even in each coordinate, we see, by considering antipodal points on the sphere, that
	\begin{equation*}
		2g_0(\omega)
		=g_0(\omega)+g_0(-\omega)
		=g_0(\xi)+g_0(-\xi)=2g_0(\xi),
	\end{equation*}
	for all $\omega,\xi\in\mathbb{S}^{d-1}$. This establishes that $g_0$ is a constant function $g_0\equiv A$, for some $A\in\mathbb{R}$.
	
	We now move on to the characterization of $g_i$, for each $i\in\{1,\ldots,d\}$. Recall that $g_i(\omega)$ is odd in $\omega_i$ and even in all other coordinates. We claim that there exists $B_i\in\mathbb{R}$ such that
	\begin{equation}\label{linear}
		g_i(\omega)=B_i\omega_i,
	\end{equation}
	for all $\omega\in\mathbb{S}^{d-1}$.
	
	For simplicity, we focus on the case $i=1$ and write $\omega=(\omega_1,\widetilde \omega)$, with $\widetilde\omega=(\omega_2,\ldots,\omega_d)$, for any $\omega\in\mathbb{S}^{d-1}$. Thus, since $g_1$ is a collision invariant which is odd in $\omega_1$ and even in $\widetilde\omega$, we deduce that
	\begin{equation*}
		\begin{aligned}
			2g_1(\omega_1,\widetilde\omega)
			&=g_1(\omega_1,\widetilde\omega)+g_1(\omega_1,-\widetilde\omega)
			\\
			&=g_1(\omega_1,\widetilde\xi)+g_1(\omega_1,-\widetilde\xi)
			=2g_1(\omega_1,\widetilde\xi),
		\end{aligned}
	\end{equation*}
	for all $(\omega_1,\widetilde\omega),(\omega_1,\widetilde\xi)\in\mathbb{S}^{d-1}$. It follows that $g_1(\omega)$ is independent of $\widetilde \omega$ and, therefore, one can write $g_1(\omega)=h_1(\omega_1)$, where $h_1:[-1,1]\to\mathbb{R}$ is an odd measurable function.
	
	Next, we consider any $s,t,u,v\in [-1,1]$ such that $s+t=u+v$. We want to show that
	\begin{equation}\label{functional:1}
		h_1(s)+h_1(t)=h_1(u)+h_1(v).
	\end{equation}
	To that end, without loss of generality, we can assume that $s\leq u\leq v\leq t$ and $s\neq t$. We are now going to construct $\sigma,\tau,\mu,\nu\in\mathbb{R}^{d-1}$ such that $(s,\sigma)$, $(t,\tau)$, $(u,\mu)$ and $(v,\nu)$ belong to $\mathbb{S}^{d-1}$, and $\sigma+\tau=\mu+\nu$. Since $d\geq 3$, a simple geometric argument readily shows that such values always exist. However, for clarity, we also provide now explicit analytical expressions for each value $\sigma$, $\tau$, $\mu$ and $\nu$, in terms of $s$, $t$, $u$ and $v$. This construction is not unique.
	
	Specifically, we first set
	\begin{equation*}
		\begin{aligned}
			\sigma&=(\sqrt{1-s^2},0,\ldots,0),
			\\
			\tau&=(\sqrt{1-t^2},0,\ldots,0),
		\end{aligned}
	\end{equation*}
	and introduce the parameter
	\begin{equation*}
		\lambda=\frac 12\left(1+\frac{v-u}{t-s}\right).
	\end{equation*}
	Note that $\frac 12\leq\lambda\leq 1$ and
	\begin{equation*}
		\begin{aligned}
			\lambda s+(1-\lambda) t&=u,
			\\
			(1-\lambda)s+\lambda t&=v.
		\end{aligned}
	\end{equation*}
	In particular, by concavity of the function $z\mapsto\sqrt{1-z^2}$, it holds that
	\begin{equation*}
		\begin{aligned}
			\lambda \sqrt{1-s^2}+(1-\lambda) \sqrt{1-t^2}&\leq \sqrt{1-u^2},
			\\
			(1-\lambda) \sqrt{1-s^2}+\lambda \sqrt{1-t^2}&\leq \sqrt{1-v^2}.
		\end{aligned}
	\end{equation*}
	This allows us to define the vectors
	\begin{equation*}
		\begin{aligned}
			\mu&=\left(\lambda\sqrt{1-s^2}+(1-\lambda)\sqrt{1-t^2},-\sqrt{1-u^2-\left(\lambda\sqrt{1-s^2}+(1-\lambda)\sqrt{1-t^2}\right)^2},0,\ldots,0\right),
			\\
			\nu&=\left((1-\lambda)\sqrt{1-s^2}+\lambda\sqrt{1-t^2},\sqrt{1-v^2-\left((1-\lambda)\sqrt{1-s^2}+\lambda\sqrt{1-t^2}\right)^2},0,\ldots,0\right).
		\end{aligned}
	\end{equation*}
	Observe that the assumption that the dimension is at least three is essential in this step.
	
	It is then straightforward to verify that
	\begin{equation*}
		s^2+|\sigma|^2=t^2+|\tau|^2=u^2+|\mu|^2=v^2+|\nu|^2=1.
	\end{equation*}
	Moreover, exploiting that $s+t=u+v$, another simple computation establishes that
	\begin{equation*}
		1-u^2-\left(\lambda\sqrt{1-s^2}+(1-\lambda)\sqrt{1-t^2}\right)^2
		=
		1-v^2-\left((1-\lambda)\sqrt{1-s^2}+\lambda\sqrt{1-t^2}\right)^2,
	\end{equation*}
	which then easily leads to the relation $\sigma+\tau=\mu+\nu$.
	
	All in all, since $g_1$ is a collision invariant and the vectors $(s,\sigma)$, $(t,\tau)$, $(u,\mu)$ and $(v,\nu)$ on the sphere satisfy the conservation of momentum \eqref{conservation:momentum}, we conclude that \eqref{functional:1} holds true. Therefore, further noticing that $h_1(0)=0$, for $h_1$ is odd, we infer that
	\begin{equation*}
		h_1(x)+h_1(y)=h_1(x+y),
	\end{equation*}
	for all $x,y\in[-1,1]$ such that $x+y\in [-1,1]$. This is the well-known Cauchy functional equation, which, by Corollary \ref{cauchy:measurable}, implies that $h_1$ is a linear function, thereby establishing \eqref{linear}.
	
	We have now completed the characterization of each term in the decomposition \eqref{decomposition:2}. In particular, introducing the vector $B=(B_1,\ldots,B_d)\in\mathbb{R}^d$ shows that $g(\omega)=A+B\cdot\omega$, which completes the proof of the theorem.
\end{proof}


\EditInfo{January 2, 2024}{March 7, 2024}{Ana Cristina Moreira Freitas, Carlos Florentino, Diogo Oliveira e Silva, and Ivan Kaygorodov}

\end{document}